\documentclass[12pt]{extarticle}
\usepackage{extsizes}
\usepackage[reqno]{amsmath} 
\usepackage{amsmath}
\usepackage{amssymb}
\usepackage{amsthm}
\usepackage{amscd}
\usepackage{amsfonts}
\usepackage{graphicx}
\usepackage{dsfont}
\usepackage{fancyhdr}
\usepackage{enumerate}
\usepackage{youngtab}
\usepackage[utf8]{inputenc}
\usepackage{comment}

\usepackage{subfigure}
\usepackage[margin=1in]{geometry}
\usepackage{graphicx}
\usepackage[noend]{algorithmic}
\usepackage{algorithm}
\usepackage{color}
\usepackage{hyperref}
\usepackage{notation}

\theoremstyle{definition}
\newtheorem{theorem}{Theorem}

\newtheorem{lemma}[theorem]{Lemma}
\newtheorem{corollary}[theorem]{Corollary}

\newcommand*\samethanks[1][\value{footnote}]{\footnotemark[#1]}
\title{Colouring the normalized Laplacian}
\author{Gabriel Coutinho\thanks{Department of Computer Science, Federal University of Minas Gerais, Belo Horizonte, Brazil. [gabriel,rafaelgrandsire,celiopassos]@dcc.ufmg.br} \and Rafael Grandsire\samethanks \and Célio Passos\samethanks}

\begin{document}

\maketitle

\begin{abstract} 
    We apply Cauchy's interlacing theorem to derive some eigenvalue bounds to the chromatic number using the normalized Laplacian matrix, including a combinatorial characterization of when equality occurs. Further, we introduce some new expansion type of parameters which generalize the Cheeger constant of a graph, and relate them to the colourings which meet our eigenvalue bound with equality. Finally, we exhibit a family of examples, which include the graphs that appear in the statement of the Erd\H{o}s-Faber-Lovász conjecture.
    \begin{center}
    \bf Keywords
    \end{center}
    Normalized Laplacian, interlacing, chromatic number, Cheeger constant.
\end{abstract}

\section{Introduction}

Let $G$ be a graph without isolated vertices, $\aA$ its adjacency matrix, and $\dD$ the diagonal matrix that records the degrees of the vertices. The matrix
\[\lL = \mathbf I - \dD^{-1/2} \aA \dD^{-1/2}\]
is known as the normalized Laplacian matrix of $G$. The standard reference for an account of the interplay between combinatorics and the spectral properties of $\lL$ is Fan Chung's book \cite{FanChungSGT}. In this paper, we are interested in which properties of $\lL$ can give information about graph colourings. It is unlikely that information contained solely in the eigenvalues of $\lL$ could determine $\chi(G)$. However one can find some eigenvalue bounds to $\chi(G)$. An eigenvalue bound is considered relatively serious if it achieves equality for at least one non-trivial infinite class of graphs, and, moreover, if the equality implies some combinatorial structure. We present a strictly speaking new bound in Theorem \ref{thm:bound}. Across the paper, we study some consequences of the equality cases in this bound, and in the last section we comment on a possible interesting class of examples connected to the Erd\H{o}s-Faber-Lovász conjecture. We also introduce in Section \ref{sec:4} new expansion-type of parameters, and compare some of their properties to the spectrum of $\lL$.

\section{Interlacing bound for the normalized Laplacian}

We begin by re-stating the famous interlacing theorem due to Cauchy (see for instance \cite[Theorem 2.1]{HaemersInterlacing}).
\begin{theorem}\label{thm:interlacing}

Let $\mathbf P$ be a real $n \times m$ matrix such that $\mathbf P^T \mathbf P = \mathbf I$ and let $\mathbf M$ be a symmetric $n \times n$ matrix with eigenvalues $\theta_1 \geq ... \geq \theta_n$. Define $\mathbf N = \mathbf P^T \mathbf M \mathbf P$ and let $\bf N$ have eigenvalues $\mu_1 \geq ... \geq \mu_m$ and respective eigenvectors $\mathbf{v}_1,...,\mathbf{v}_m$. Then the following hold.

\begin{enumerate}[(i)]

\item The eigenvalues of $\bf N$ interlace those of $\bf M$, meaning that, for $i = 1,...,m$, we have
\[\theta_{n-m+i} \leq \mu_i \leq \theta_i.\]

\item If $\mu_i = \theta_i$ or $\mu_i = \theta_{n-m+i}$ for some $i \in [1,m]$, then $\nN$ has a $\mu_i$-eigenvector $\vv$ so that $\pP \vv$ is a $\mu_i$-eigenvector of $\mM$.

\item If, for some integer $\ell$, $\mu_i = \theta_i$ for $i = 1,...,\ell$ (or $\mu_i = \theta_{n-m+i}$ for $i = \ell,...,m$), then $\mathbf P \vv_i$ is a $\mu_i$-eigenvector for $\bf M$ for $i = 1,...,\ell$ (respectively $i = \ell,...,m$).

\item If the interlacing is tight, that is, if there is an integer $k \in [0,m]$ so that 
\[\theta_i = \mu_i\quad \text{for $1 \leq i \leq k$,} \quad \text{and} \quad  \theta_{n-m+i} = \mu_i \quad \text{for $k+1 \leq i \leq m$},\]
then
\[\mM \pP = \pP \nN.\]
\end{enumerate} \qed

\end{theorem}

Let $G$ be a graph on $n$ vertices. Let $\mathbf A(G) = \bf A $ be its adjacency matrix, and $\bf D$ the diagonal matrix that records the degree sequence of $G$. The matrix $\bf D- \bf A$ is usually called the Laplacian matrix of $G$, but in this paper we are interested in its normalized version
\[\mathbf {L} = \mathbf{I} - \mathbf {D}^{-1/2} \mathbf A \mathbf D^{-1/2}.\]
Say $G$ contains a (proper) $k$-colouring, that is, its vertex set can be partitioned into $k$ independent sets $S_1,...,S_k$. Let $\pi$ be this partition and $\mathbf S$ be its characteristic matrix, that is, the $01$ matrix whose rows are indexed by the vertices, columns by the colour classes, and an entry equal to $1$ indicates that the vertex corresponding to the row has been coloured by the class corresponding to the column. By $G / \pi$ we denote the quotient graph, meaning, the graph whose vertices are the classes of $\pi$ and between any two vertices there is a weighted edge whose weight is equal to the number of edges between the corresponding classes in $G$. In particular, $G / \pi$ has no loops because $\pi$ is a (proper) colouring, and $\mathbf S^T \mathbf A \mathbf S = \mathbf A(G / \pi)$. 

We define
\[\mathbf P = \mathbf D^{1/2} \mathbf S \mathbf N^{-1/2},\]
where $\bf N$ is the diagonal matrix whose $k$th diagonal entry is equal to sum of the degrees of the vertices in $S_k$. In particular, this turns the columns of $\mathbf P$ into normalized vectors.

Thus
\begin{align} \label{eq:PTP=I}
\mathbf P^T \mathbf P = \mathbf I.
\end{align}
Moreover, 
\begin{align} \label{eq:interlacing}
\mathbf P^T  \mathbf L \mathbf P = \mathbf I - \mathbf N^{-1/2} \mathbf A(G/ \pi) \mathbf N^{-1/2} = \mathbf L(G / \pi).
\end{align}
Observe that $0$ is an eigenvalue of both $\bf L$ and $\mathbf{L}(G / \pi)$ with respective eigenvectors $\mathbf D^{1/2} \mathbf{1}$ and $\mathbf N^{1/2} \mathbf{1}$. If $0 = \lambda_0 \leq \lambda_1 \leq ... \leq \lambda_{n-1}$ are the eigenvalues of $\bf L$, then we define
\[\sigma_k = \sum_{j = 1}^k \lambda_{n-j},\]
that is, the sum of the largest $k$ eigenvalues of $\bf L$.
\begin{theorem}\label{thm:bound}
If $\sigma_{k-1} < k$, then $k < \chi(G)$.
\end{theorem}
\begin{proof}
This is a consequence of Equations \eqref{eq:PTP=I} and \eqref{eq:interlacing} and Theorem \ref{thm:interlacing}. In particular, let $k \geq \chi(G)$, and consider a partition $\pi$ of $V(G)$ into $k$ independent sets.  The eigenvalues of $\mathbf L(G / \pi)$ interlace those of $\bf L$, so
\[k = \tr \mathbf L(G / \pi) \leq \sigma_{k}.\]
We can however ignore the least eigenvalue of $\mathbf L(G / \pi)$ which is $0$, thus the inequality still holds for $\sigma_{k-1}$.
\end{proof}

Let $\pi = \{S_1,...,S_k\}$ be any $k$-partition of $V(G)$. Then the entries of $\aA(G /\pi)$ record the number of edges between the classes of $\pi$. We recall the usual notation $e(S,T)$ that counts the number of edges between sets $S$ and $T$, and $\Vol (S)$ which is the sum of the degrees of the vertices in $S$. Thus
\[\tr \mathbf{L} (G /\pi) = k - \sum_{j = 1}^k \frac{e(S_j,S_j)}{\Vol (S_j)}.\]
Using interlacing, we reach an interesting lower bound on what any attempt of colouring a graph with $k$ colours can yield.
\begin{corollary} \label{cor:1}
For any $k$-partition $\pi = \{S_1,...,S_k\}$ of a graph $G$, it follows that
\[\sum_{j = 1}^k \frac{e(S_j,S_j)}{\Vol(S_j)} \geq k - \sigma_{k-1}.\] \qed
\end{corollary}

We have two remarks.

First, Theorem \ref{thm:bound} indicates that graphs with small chromatic number should have large normalized Laplacian eigenvalues. For instance, the well-known fact that bipartite graphs have largest eigenvalue equal to $2$ follows immediately from the result above.

Second, Theorem \ref{thm:bound} has been, \textit{mutatis mutandis}, already known for a long time. Interlacing has seen countless applications in graph theory. As far as we know, first versions of this technique to study the chromatic number appeared in Haemers \cite{HaemersEigenvalueMethods}, where interlacing for the adjacency matrix was used. In fact, Theorem \ref{thm:bound} is equivalent to \cite[Proposition 3.6.3 (i)]{BrouwerHaemers} for regular graphs. More recently, Bollobás and Nikiforov obtained similar results for the Laplacian matrix \cite{BollobasNikiforovInterlacing,BollobasNikiforovExact}. Finally, Butler used interlacing \cite[Chapter 5]{ButlerPhdThesis} for the normalized Laplacian and proved almost the same result (Theorem 36) as Corollary \ref{cor:1} above. The only, but yet relevant, difference is that we have our results for $\sigma_{k-1}$, and not $\sigma_k$. This seems a small difference, but note that the results are trivially true if stated with $\sigma_k$ instead.

Our main goal over the next sections is to try to provide some new useful insights about an old technique.

\section{Tight interlacing and equitable colourings} \label{sec:3}

Recall that $G$ is a graph with normalized Laplacian matrix $\bf L$, and we denote its eigenvalues by $0 = \lambda_0 \leq ... \leq \lambda_{n-1}$. We also defined $\sigma_k(G)$ as the sum of the $k$ largest eigenvalues of $\mathbf L (G)$, and made a similar definition for $G / \pi$ replacing $G$. Assume $\pi= \{S_1,...,S_k\}$ is a $k$-colouring with characteristic matrix $\bf S$. Again, $\bf D$ is the diagonal matrix recording the degrees of $G$ and $\mathbf N$ is the diagonal matrix recording the sum of the degrees of the vertices in each $S_j$. We now turn our attention to the case where interlacing (see Theorem \ref{thm:interlacing}) is tight.

\begin{lemma}\label{thm:tight}
Let $\pi$ be a $k$-colouring and $\mathbf P = \mathbf D^{1/2} \mathbf S\mathbf N^{-1/2}$. If $\sigma_{k-1} = k$, then interlacing of $\mathbf L = \mathbf L(G)$ according to $\bf P$ is tight, and in particular, $\mathbf L \mathbf P = \mathbf P \mathbf L(G / \pi)$.
\end{lemma}
\begin{proof}
If $\sigma_{k-1}(G) = k$, then
\[k = \sigma_{k-1}(G) \geq \sigma_{k-1}(G / \pi) = \tr \mathbf L(G /\pi) = k.\]
So the inequality holds with equality, and therefore each eigenvalue of $\mathbf L(G /\pi)$ must meet the upper bound given by interlacing, except for the last. Hence interlacing is tight, and from Theorem \ref{thm:interlacing} item (iv), the equality $\mathbf L \mathbf P = \mathbf P \mathbf L(G / \pi)$ holds.
\end{proof}

Consider a partition $\pi = \{S_1,...,S_m\}$ of the set $[n]$. Assume this partition has characteristic matrix $\qQ$ (which is $n \times m$). Let $\bf M$ be a symmetric $n \times n$ matrix. The partition $\pi$ is called \textit{equitable} to $\bf M$ if
\[\sum_{k \in S_j} \mathbf M_{ik}\]
is constant for all $i$ which belongs to the same class of $\pi$, for all $j$, or, equivalently, if the column space of $Q$ is $M$-invariant. If $\bf M = \bf A$, then to say that $\pi$ is equitable to $\bf A$ is equivalent to saying that the number of neighbours a vertex in class $i$ has in class $j$ depends only on $i$ and $j$.

The theorem below is an adaptation of the theory of equitable partitions for adjacency matrices \cite[Chapter 9]{GodsilRoyle} to the normalized Laplacians.
\begin{lemma}\label{thm:equitable}
Let $\pi$ be a partition of $V(G)$ with characteristic matrix $\mathbf S$, and let $\mathbf P = \mathbf D^{1/2} \mathbf S \mathbf N^{-1/2}$. Then $\pi$ is equitable to $\mathbf{D}^{-1} \mathbf A$ if and only if $\mathbf L(G) \mathbf P = \mathbf P \mathbf L(G/\pi)$.
\end{lemma}
\begin{proof}
We use $e(S,T)$ to denote the number of edges between sets of vertices $S$ and $T$. If $S = \{a\}$, we might simply write $e(a,T)$. First we shall note that $\mathbf L(G) \mathbf P = \mathbf P \mathbf L(G/\pi)$ is equivalent to 
\[ \mathbf{D}^{-1} \mathbf A(G) \mathbf S = \mathbf S \mathbf N^{-1} \mathbf A(G/\pi).\]
If this equality holds, and if $a$ is a vertex that belongs to class $S_i$, then
\[\frac{e(a,S_j)}{d(a)} = (\mathbf{D}^{-1} \mathbf A(G) \mathbf S )_{aj} = (\mathbf S \mathbf N^{-1} \mathbf A(G/\pi))_{aj} = \frac{e(S_i,S_j)}{\Vol(S_i)}.\]
So the partition is equitable. On the other hand, if the partition is equitable, then fix $S_i$, say $S_i = \{a_1,...,a_f\}$, and for all $S_j$, it follows that for all $a_k,a_\ell \in S_i$, we have
\[\frac{e(a_k,S_j)}{d(a_k)} = \frac{e(a_\ell,S_j)}{d(a_\ell)}.\]
So there are constants $A,B$ and $q_1,...,q_f$ such that $e(a_k,S_j) = A q_k$ and $d(a_k) = Bq_k$. Then it is immediate to check that for all $a \in S_i$,
\[\frac{e(S_i,S_j)}{\Vol(S_i)} = \frac{A \sum q_k}{ B \sum q_k} = \frac{A}{B} = \frac{e(a,S_j)}{d(a)},\]
thus $\mathbf L(G) \mathbf P = \mathbf P \mathbf L(G/\pi)$.
\end{proof}

Using Theorem 1 from \cite{NikiforovChromaticSpectralRadius}, the following lemma is straightforward. We provide a proof that follows immediately from Theorem \ref{thm:bound}.
\begin{lemma}\label{lem:largesteval}
	Let $\lambda = \lambda_{n-1}$ be the largest eigenvalue of the normalized Laplacian matrix $\bf L$ of a graph $G$. Then
	\[\chi \geq 1 + \frac{1}{\lambda - 1}.\]
\end{lemma}
\begin{proof}
    If $k = \chi$,  then $\chi \le \sigma_{\chi-1}(G) \le (\chi - 1)\lambda$, which implies the result.
\end{proof}

We can now provide a significant consequence for the case where equality in the lemma above holds.

\begin{theorem}\label{thm:3}
	Let $G$ be a graph which admits a colouring with $k$ colours, and let $\lambda = \lambda_{n-1}$ be its largest normalized Laplacian eigenvalue. If
	\[\lambda = \frac{k}{k-1},\]
	then this colouring is equitable with respect to $\dD^{-1} \aA$. 
\end{theorem}
\begin{proof}
	If $\pi$ is a $k$-colouring, we have 
	$k = (k-1)\lambda \ge \sigma_{k-1}(G) \ge \sigma_{k-1}(G/\pi) = k$.
	
	In this case, $\sigma_{k-1}(G) = k$, and by Lemma \ref{thm:tight} we have a tight interlacing according to $\pi$, resulting in the equality $\mathbf{LP} = \mathbf{PL}(G/\pi)$. Therefore, by Lemma \ref{thm:equitable}, $\pi$ must be equitable with respect to $\mathbf{D}^{-1}\mathbf{A}$.
\end{proof}

 \section{Regular colourings} \label{sec:4}

We say that a colouring of a graph $G$ is \textit{regular} if between any two colour classes there is the same number of edges.  For a partition $\pi = \{S_1, \dots, S_m\}$ of $V(G)$, define
\begin{enumerate}[(i)]
\item $\displaystyle \gamma(\pi) = \frac{\min_{i \neq j} e(S_i, S_j)}{\max_i \Vol S_i},$
\item $\displaystyle \gamma^*(\pi) = \frac{\max_{i \neq j} e(S_i, S_j)}{\min_i \Vol S_i}.$
\end{enumerate}
Moreover, for reasons that will become clear quite soon, define
\begin{enumerate}[(i)]
\item[(iii)] $\displaystyle \psi_k(G) = \max_{\pi, \text{$k$-partition}} \gamma(\pi)$,
\item[(iv)]  $\displaystyle \phi_k(G) = \min_{\pi, \text{$k$-partition}} \gamma^*(\pi)$.
\end{enumerate}

\begin{lemma}\label{lem:regcol}
	If $\pi$ is a $k$-partition of $V(G)$, then 
	\[\gamma(\pi) \le \frac{1}{k-1},\]
	with equality if and only if $\pi$ is a regular coloring.
\end{lemma}
\begin{proof}
	\begin{align*}
	\gamma(\pi) &= \frac{\min_{i \neq j} e(S_i, S_j)}{\max_i \Vol S_i} \leq \frac{\frac{1}{\binom{k}{2}}\left(\frac{\Vol G}{2}\right)}{\max_i \Vol S_i}\\
	&\le \frac{\frac{1}{\binom{k}{2}}\left(\frac{\Vol G}{2}\right)}{\frac{1}{k} \Vol G} = \frac{1}{k-1}.
	\end{align*}
	The first inequality follows from assuming that all edges of the graph are equally distributed between the classes, and if this holds with equality, then the second inequality (which is immediate) also holds with equality.
\end{proof}

\begin{lemma}\label{lem:regcol2}
	If $\pi$ is a $k$-colouring of $G$, then
	\[\gamma^*(\pi) \geq \frac{1}{k-1},\]
	with equality if and only if $\pi$ is a regular colouring.
\end{lemma}
\begin{proof}
	The proof is very similar to that of Lemma \ref{lem:regcol}.
\end{proof}

It is fair to say that the larger $\gamma(\pi)$ is, the better $\pi$ is as an attempt to regularly colour $G$. Likewise, the smaller $\gamma^*(\pi)$ is, the more $\pi$ looks like a clustering with clusters of roughly the same number of edges. These comments motivate the definitions of $\psi_k(G)$ and $\phi_k(G)$. Note that $\phi_2(G)$ is indeed the so called Cheeger constant of $G$.

\begin{lemma} \label{lem:inequalities}
	Let $\pi$ be a $k$-partition of $V(G)$, $G/\pi$ be the quotient graph, and $0 = \theta_0 \leq \theta_1 \leq ... \leq \theta_{k-1}$ be the eigenvalues of $\lL = \lL(G/\pi)$. Then
	\[k\gamma(\pi) \leq \theta_1 \quad \text{and} \quad  \theta_{k-1} \leq k\gamma^*(\pi).\] 
\end{lemma}
\begin{proof}
	Recall that $\nN$ is the diagonal matrix recording the row sums of $\mathbf{A}(G/\pi)$. From the Rayleigh quotient expression for the eigenvalues and the Courant-Fisher Theorem (see for instance \cite[Chapter 2]{BrouwerHaemers}), it follows that
	\begin{align*}
	\theta_1 & = \min_{\vv \perp \nN \cdot \mathbf 1} \frac{(\nN^{1/2}\vv)^T \lL (\nN^{1/2}\vv) }{(\nN^{1/2}\vv)^T (\nN^{1/2}\vv)} \\ &  = \min_{\vv \perp \nN \cdot \mathbf 1} \frac{\sum_{i,j} (\vv_i - \vv_j)^2 e(S_i,S_j)}{\sum_{i} \vv_i^2 \Vol S_i} 
	\\ &  \geq \min_{\vv \perp \mathbf 1} \frac{\sum_{i,j} (\vv_i - \vv_j)^2 e(S_i,S_j)}{\sum_{i} \vv_i^2 \Vol S_i}
	\\ & \geq \gamma(\pi) \min_{\vv \perp \mathbf 1} \frac{\sum_{i,j} (\vv_i - \vv_j)^2}{\sum_{i} \vv_i^2}\\ & = \gamma(\pi) k,
	\end{align*}
	where the last equality follows from the fact that all non-zero eigenvalues of the conventional Laplacian matrix $(k-1)\mathbf I - \aA(K_k)$ of the complete graph on $k$ vertices are equal to $k$. The other inequality from the statement follows analogously.
\end{proof}

\begin{theorem}\label{thm:4}
Let $G$ be graph. Fix $k$ between $1$ and $n$. Recall the definitions of $\psi_k(G)$ and $\phi_k(G)$ from the beginning of this section. Again, let $\lambda_{0} \leq\lambda_1 \leq ... \leq \lambda_{n-1}$ be the eigenvalues of $\lL(G)$, the normalized Laplacian matrix of $G$. Then
    \begin{enumerate}[(i)]
        \item $\psi_k(G)k \le \lambda_{n-(k-1)}$, and
        \item $\phi_k(G)k \ge \lambda_{k-1}$.
    \end{enumerate}
\end{theorem}
\begin{proof}
    This is an immediate consequence of Lemma \ref{lem:inequalities} and interlacing.
\end{proof}

\begin{corollary}
 Suppose $G$ has a regular $k$-colouring, and $\sigma_{k-1} = k$. Then $\lambda_{n-1} = k/(k-1)$.
\end{corollary}
\begin{proof}
From Lemma \ref{lem:regcol}, it follows that $\psi_k(G) = 1/(k-1)$. From Theorem \ref{thm:4}, we have $k/(k-1) \leq \lambda_{n-(k-1)}$, which, in addition to $\sigma_{k-1} = k$, leads to $\lambda_{n-1} = k/(k-1)$.
\end{proof}

The concept of regular colourings also interacts nicely with the results in the previous section.

\begin{theorem}\label{thm:regular}
    Suppose $\lambda_{n-1} = k/(k-1)$ and that $G$ admits $k$-colouring. Then any $k$-colouring of $G$ is regular.
\end{theorem}
\begin{proof}
From Theorem \ref{thm:bound}, it follows that $\sigma_{k-1} \geq k$, and thus
\[\lambda_{n-(k-1)} = \cdots = \lambda_{n-1} = \frac{k}{k-1}.\]
As a consequence, just like we argued in Theorem \ref{thm:3}, it follows that all eigenvalues of the quotient graph are equal to $k/(k-1)$, except for the $0$ eigenvalue. It is straightforward to verify that the quotient graph is a complete graph where all edges have the same edge weight.
\end{proof}

\begin{corollary}\label{cor:divis}
 Suppose $\lambda_{n-1} = k/(k-1)$ and that $G$ admits $k$-colouring. Then any vertex $a$ has precisely $d(a)/(k-1)$ neighbours of each colour different from its own.
\end{corollary}
\begin{proof}
Let $\pi$ be a $k$-colouring. From Theorem \ref{thm:3} and Theorem \ref{thm:regular}, we know that $\pi$ is equitable with respect to $\dD^{-1} \aA$ and is a regular colouring. That is, there are constants $\alpha$ and $\beta$ such that
\[\lL(G) \pP = \pP (\alpha \mathbf I + \beta \mathbf J),\]
where $\mathbf J$ is the all $1$s matrix. From this it follows that the rows of $\lL(G) \pP$ are constant, with the exception of the entry in the column that corresponds to the class that contains the vertex of the row. This is equivalent to what we wanted to prove.
\end{proof}

Note that Corollary \ref{cor:divis} offers a quite strong integrality condition on the degrees of all vertices of a graph that satisfy Lemma \ref{lem:largesteval} with equality.

\section{$m$-uniform linear hypergraphs}

A natural question at this point is whether there are examples of graphs whose colourings are equitable with respect to $\dD^{-1} \aA$ or regular. A \textit{hypergraph} is a graph in which edges are allowed to have more than $2$ vertices. A uniform hypergraph has all edges with the same number of vertices. A colouring of the vertex set of a hypergraph is strong if the vertices in each edge all have distinct colours. Finally, a linear hypergraph is such that any two edges intersect in at most one vertex. The standard reference is \cite{BergeGraphHypergraphs}. Linear hypergraphs appear naturally in connection to partial linear spaces.

Given an $m$-uniform linear hypergraph, one forms the underlying graph by replacing each edge by a clique on $m$ vertices. Thus $m$-uniform linear hypergraphs correspond to graphs where all vertices belong to cliques of size $m$, and no two of these have an edge in common. Thus, strongly colouring the hypergraph or properly colouring its underlying graph amounts to the same task. The topic of colouring such graphs has received a considerable amount of attention because of the well-known Erd\H{o}s-Faber-Lovász conjecture, which proposes that any $m$-uniform linear hypergraph on $m$ edges is $m$-strongly colourable (it obviously cannot be coloured with fewer colours). Now we observe that any $m$-strong colouring of an $m$-uniform linear hypergraph on $e$ hyperedges (or any $m$-colouring of its underlying graph $G$) must necessarily be
\begin{itemize}
    \item equitable with respect $\dD(G)^{-1} \aA(G)$, as for any fixed colour, the number of neighbours of this colour that a vertex has is proportional to the its degree; and
    \item regular, as each colour class has precisely $e$ edges in $G$ towards any other colour class.
\end{itemize}
Indeed, we can show that both facts above relate quite well with the topic of this paper.
\begin{lemma}\label{lem:uniformlinearhypergraph}
Assume $G$ is the underlying graph of an $m$-uniform linear hypergraph. Then the largest eigenvalue $\lambda_{n-1}$ of $\lL(G/\pi)$ satisfies
\[\lambda_{n-1} \leq \frac{m}{m-1}.\]
\end{lemma}
\begin{proof}
Let $K_{(1)},\hdots,K_{(e)}$ denote the cliques of $G$. Then, for any $\vv \in \R^{V(G)}$, the Rayleigh quotient of $\ww = \dD^{1/2} \vv$ gives
\begin{align*}
\frac{\ww^T \lL \ww}{\ww^T\ww} & = \frac{\sum_{ab \in E(G)}(\vv_a-\vv_b)^2}{\sum_{a\in V(G)} \vv_a^2\ d(a)} \\
& = \frac{\sum_{\ell = 1}^e \sum_{ab \in E(K_{(\ell)})}(\vv_a-\vv_b)^2 }{\sum_{\ell = 1}^e \sum_{a \in V(K_{(\ell)})} \vv_a^2\ d_{K_{(\ell)}}(a)} \\
& \leq \max_{\ell} \frac{\sum_{ab \in E(K_{(\ell)})}(\vv_a-\vv_b)^2}{\sum_{a \in V(K_{(\ell)})} \vv_a^2\ d_{K_{(\ell)}}(a)} \\
& \leq \max_{\ell} \lambda_{m-1}(K_{(\ell)}) \\
& = \frac{m}{m-1}.
\end{align*}
\end{proof}

As a consequence of Lemma \ref{lem:largesteval}, any underlying graph of an $m$-uniform linear hypergraph that can be $m$ colourable must meet the bound above with equality. We can however show directly that $\lambda_{n-1}(\lL(G))$ is equal to $m/(m-1)$ in several cases, including those coming from the Erd\H{o}s-Faber-Lovász conjecture.
\begin{theorem}
Let $G$ be a graph on $n$ vertices, and assume $G$ is the underlying graph of an $m$-uniform linear hypergraph on $e$ hyperedges. Assume $n-e > 0$. Then the largest eigenvalue $\lambda_{n-1}$ of $\lL(G/\pi)$ satisfies
\[\lambda_{n-1} = \frac{m}{m-1}.\]
\end{theorem}
\begin{proof}
In $\R^n$, consider the span of vectors which are constant at each of the cliques. This has dimension at most $e$, the number of cliques. Hence the space of vectors which sum to zero on all cliques has dimension at least $n-e$. Let $\vv$ be one such vector, and, as before, $\ww = \dD^{1/2}\vv$. Then
\begin{align*}
\frac{\ww^T \lL \ww}{\ww^T\ww} & = \frac{\sum_{ab \in E(G)}(\vv_a-\vv_b)^2}{\sum_{a\in V(G)} \vv_a^2\ d(a)} \\
& = \frac{\sum_{\ell = 1}^e \sum_{ab \in E(K_{(\ell)})}(\vv_a-\vv_b)^2 }{\sum_{\ell = 1}^e \sum_{a \in V(K_{(\ell)})} \vv_a^2\ d_{K_{(\ell)}}(a)} \\
& \geq \min_{\ell} \frac{\sum_{ab \in E(K_{(\ell)})}(\vv_a-\vv_b)^2}{\sum_{a \in V(K_{(\ell)})} \vv_a^2\ d_{K_{(\ell)}}(a)} \\
& = \lambda_{1}(K_m) = \hdots = \lambda_{m-1}(K_m) \\
& = \frac{m}{m-1}.
\end{align*}
\end{proof}

\section{Concluding remarks}

A natural question is how the bound in Theorem \ref{thm:bound} compares to other eigenvalue bounds. According to our computations, in most cases it tends to be better than Hoffman's bound but slightly worse than the bound in \cite[Proposition 3.6.3 (i)]{BrouwerHaemers} due to Haemers. However, it seems that there are several cases in which our bound performs better, including many of the underlying graphs of uniform linear hypergraphs.

Another line of investigation is to design a spectral algorithm that attempts to find partitions realizing $\phi_k$ or $\psi_k$, and further provide tight approximation guarantees.

\bibliographystyle{plain}
\bibliography{bibevals.bib}

\end{document}